\documentclass[reqno]{amsart}
\usepackage{subfigure}
\usepackage[margin=3.5cm]{geometry}
\usepackage{color}
\usepackage{graphicx}
\usepackage{enumerate}
\usepackage{xcolor}
\usepackage{mathtools,amsthm,amssymb, amsmath, amssymb, amsfonts, epsfig}
\usepackage{setspace, stmaryrd, verbatim, euro, enumerate}
\usepackage{bbm, tensor,comment,textcomp}
\usepackage{bm}
\usepackage{appendix}
\usepackage{multirow}

\usepackage{colortbl}

\newtheorem{theorem}{Theorem}[section]
\newtheorem{corollary}[theorem]{Corollary}
\newtheorem{lemma}[theorem]{Lemma}
\newtheorem{proposition}[theorem]{Proposition}

\numberwithin{equation}{section}

\theoremstyle{definition}
\newtheorem{definition}[theorem]{Definition}
\newtheorem{assumption}[theorem]{Assumption}

\newtheorem{remark}[theorem]{Remark}

\theoremstyle{plain}
\newtheorem{example}[theorem]{Example}

\definecolor{winered}{rgb}{0.6,0,0}
\definecolor{ocean}{rgb}{0,0.1,0.6}
\definecolor{forest}{rgb}{0,0.5,0.1}
\definecolor{fill}{rgb}{.92,1,.92}
\definecolor{sunset}{rgb}{.9,0.4,0}
\def\blue#1{\textcolor{blue}{#1}}

\usepackage[colorlinks=true, linkcolor=ocean,citecolor=winered,urlcolor=winered]{hyperref}

\newcommand{\norm}[1]{\left\lVert#1\right\rVert}
\newcommand{\norminf}[1]{\left\lVert#1\right\rVert_\infty}
\newcommand{\abs}[1]{\left\lvert#1\right\rvert}

\newcommand\cA{\mathcal{A}}
\newcommand\cB{\mathcal{B}}
\newcommand\cC{\mathcal{C}}
\newcommand\cD{\mathcal{D}}
\newcommand\cE{\mathcal{E}}
\newcommand\cF{\mathcal{F}}
\newcommand\cG{\mathcal{G}}

\newcommand\cM{\mathcal{M}}

\newcommand\cP{\mathcal{P}}

\newcommand\cW{\mathcal{W}}
\newcommand\cX{\mathcal{X}}

\newcommand{\bsigma}{\boldsymbol \sigma}
\newcommand{\bb}{\boldsymbol b}

\newcommand\bE{\mathbb{E}}

\newcommand\bN{\mathbb{N}}

\newcommand\bR{\mathbb{R}}

\newcommand{\D}{\mathrm{d}}
\newcommand{\E}{\mathrm{e}}
\newcommand{\BB}{\mathrm{B}}

\newcommand{\half}{\frac{1}{2}}

\newcommand{\ep}{\varepsilon}
\newcommand{\vphi}{\varphi}
\newcommand{\notthis}[1]{}
\newcommand{\one}{\mathbbm{1}}

\newcommand{\olambda}{\overline{\lambda}}

\newcommand{\dt}{\, \mathrm{d}t}
\newcommand{\ds}{\, \mathrm{d}s}
\newcommand{\du}{\, \mathrm{d}u}
\newcommand{\dr}{\, \mathrm{d}r}
\newcommand{\dx}{\, \mathrm{d}x}

\newcommand{\bbl}{\bm{\overline{\lambda}}}

\usepackage{charter} 
\usepackage[foot]{amsaddr} 

\usepackage[textsize=tiny]{todonotes}

\begin{document}

\title{On the large-time behaviour of affine Volterra processes}

\author{Antoine Jacquier}
\address[Antoine Jacquier]{Department of Mathematics, Imperial College London and the Alan Turing Institute, UK.}
\email{a.jacquier@imperial.ac.uk}

\author{Alexandre Pannier$^\ast$}
\address[Alexandre Pannier]{LPSM, Université Paris Cité}
\email{pannier@lpsm.paris}

\author{Konstantinos Spiliopoulos}
\address[Konstantinos Spiliopoulos]{Boston University, Department of Mathematics and Statistics\\ 111 Cummington Mall, Boston, MA 02215, USA}
\email{kspiliop@bu.edu}

\maketitle


\begin{abstract}
    We show the existence of a stationary measure for a class of multidimensional stochastic Volterra systems of affine type. These processes are in general not Markovian, a shortcoming which hinders their large-time analysis. We circumvent this issue by lifting the system to a measure-valued stochastic evolution equation introduced by Cuchiero and Teichmann~\cite{CT18}, whence we retrieve the Markov property. Leveraging on the associated generalised Feller property, we extend the Krylov-Bogoliubov theorem to this infinite-dimensional setting and thus establish an approach to the existence of invariant measures. We present concrete examples, including the rough Heston model from Mathematical Finance.
\end{abstract}



\section{Introduction}
We are interested in the large-time behaviour of multidimensional Stochastic Volterra Equations (SVEs) of affine type which take the form
\begin{equation}\label{eq:SVE}
V_t = V_0 + \int_0^t K(t-s) \bb(V_s)\ds + \int_0^t K(t-s)\bsigma(V_s)\,\D W_s, \quad t\ge0,
\end{equation}
where $\bsigma\bsigma^\top:\bR^d\to\bR^{d\times d}$ and $\bb:\bR^d\to\bR^d$ are affine. Define~$\bR_+:=[0,\infty)$; the kernel~$K\in L^2_{{\rm loc}}(\bR_+, \bR^{d\times d})$ is a diagonal matrix with each component being the Laplace transform of some signed measure on~$\bR_+$, such as the kernel~$t\mapsto\frac{t^{\alpha-1}}{\Gamma(\alpha)}\E^{-\delta t},\,\alpha\in(1/2,1],\delta>0$.

The main result of this paper is the existence of an invariant measure for the Markovian lift of~\eqref{eq:SVE} introduced by Cuchiero and Teichmann~\cite{CT18}; this entails a new notion of stationarity for~$V$. The renewed interest in SVEs stems from the emergence of a new class of stochastic volatility models which surrender the comfort of Markovianity for the sake of consistency with the data. Two examples in this direction are the rough Heston model, which arises as the scaling limit of a high-frequency model governed by Hawkes processes~\cite{EFR18}, and the rough Bergomi model which appeared in~\cite{BFG16}. The former is affine in the sense of~\eqref{eq:SVE} while the latter is not.
The term affine Volterra process was coined in~\cite{ALP17} where general existence and uniqueness are derived. We note here however that despite the large number of asymptotic results, very little is known about the ergodic behaviour of rough volatility models.

In fact, to the best of our knowledge, only three related results exist in this direction. The first one is a large deviation principle for the rescaled log-price under the rough Heston model~\cite{FGS19}; it is derived from computing the limit of its characteristic function, which is known in semi-closed form~\cite{ER19}. The second one is~\cite{GR20}, where the authors proved the existence of an invariant measure for the asset price under the assumption that the non-Markovian variance process has a stationary distribution. Upon completion of this work, we became aware of the very recent (and the third known related result to us)~\cite{FJ22}.
There, using completely different methods than ours, the authors prove the existence and characterise the properties of the invariant measure associated to a special case~\eqref{eq:SVE}, where the diffusion matrix~$\bsigma$ is taken to be diagonal (so it does not include the rough Heston model for example). While our results currently only provide conditions for existence of an invariant measure (see however the discussion in Section~\ref{S:Outlook}), our method is generally applicable to SVE models as long as one can verify that the lift is a generalised Feller process and that certain bounds hold and importantly we note that it parallels the Markovian theory. Other notable results in related directions are~\cite{EBDK22,CT19}.

When the process is Markovian, one gains access to transition semigroups, which in turn allow to study the ergodic behaviour. A multitude of tools have been developed to prove existence and uniqueness of invariant measures, as well as asymptotic stability and rates of convergence. The most prominent in the literature are based either on showing that the transition semigroup has the (strong) Feller property or on dissipativity methods, which are related to Lyapunov function techniques. We refer to~\cite{DPZ92,DPZ96} for an overview of these approaches.

A  different point of view is given by the theory of Random Dynamical Systems where Markovianity is replaced by a cocycle property for the driving noise, e.g. (fractional) Brownian motion~\cite{Arnold98,GAKN09}. Furthermore, in this context one analyses the asymptotic behaviour by taking the starting time to~$-\infty$ and looking at the system at~$t=0$, instead of starting at~$t=0$ and looking at~$t$ goes to $+\infty$.

This inspired Hairer's theory of Stochastic Dynamical Systems, which aimed at studying the large-time behaviour of an SDE driven by additive fractional noise~\cite{Hairer02}. To reconcile the Markovian ergodic theory with non-Markovian processes, his main idea was to augment the state space with all the past fractional noise, and to define a Feller semigroup on this augmented space. This led to further advances with multiplicative noise in the case~$H>1/3$~\cite{HO07,HP11}, and for the discretised version of the SDE~\cite{Varvenne19}.

The more recent literature on large-time behaviour of fractional processes has flourished under the umbrella of rough paths theory, in particular regarding multiscale systems~\cite{CPKMZ19,GL20,GLS21,LS20,PIX20b,PIX20a}. Their asymptotic properties were also investigated thanks to Malliavin calculus~\cite{BGS19} and the stochastic sewing lemma~\cite{HL20}. As usual with fractional stochastic integrals however, these frameworks do not accommodate for the highly irregular paths
(Hurst exponent in $(0,1/4)$) found in rough volatility models.

The idea of~\cite{Hairer05} with augmenting the state space to recover Markovianity made its way to rough volatility modelling, albeit with a mild adaptation. Motivated by hedging applications in the rough Heston model, El Euch and Rosenbaum~\cite{ER18} showed that since the variance process is non-Markovian one needs to include the whole forward variance curve; this suggests to consider the system~$(S_t, (\bE[V_{s+t}|\cF_t])_{s\ge0})$.
The forward variance even satisfies a Stochastic Partial Differential Equation (SPDE)~\cite{JE18a}, which allowed to further characterise the Markovian structure of the model.
A similar \textcolor{black}{stochastic evolution equation} lift for another singular SVE was in fact derived in an earlier work~\cite{MS15}. The analysis of the rough Heston model is facilitated by its affine structure, unlike the rough Bergomi model, its main competitor. Yet, the forward variance curve of the latter also satisfies a \textcolor{black}{stochastic evolution equation} as it belongs to the realm of polynomial processes~\cite{CS21}. The most conducive \textcolor{black}{stochastic evolution equation} lift for our purposes, though, was introduced by Cuchiero and Teichmann in~\cite{CT18}. Unlike the previous approaches, they argue it is more sensible to start from the Markovian lift before solving the SVE.
More precisely, they consider the one-dimensional \textcolor{black}{measure-valued stochastic evolution equation}
\begin{equation}\label{5eq:lift}
\D \lambda_t(\D x) = -x \lambda_t(\D x) \dt + \nu(\D x) \D X_t,
\end{equation}
where~$\lambda_0$ and~$\nu$ are signed measures on $[0,+\infty]$ \[
X_t = -\beta \int_0^t \langle 1, \lambda_s\rangle \ds + \sigma\int_0^t  \sqrt{\langle 1, \lambda_s\rangle} \D W_s,
\]
where $\langle y, \lambda \rangle:=\int_0^\infty y(x)\,\lambda(\D x)$ for all~$y\in \cC_b(\bR_+)$.
In fact, one should interpret the \textcolor{black}{stochastic evolution equation}~\eqref{5eq:lift} in the mild sense such that
\begin{align*}
\langle y,\lambda_t \rangle = \int_0^\infty \E^{-tx} y(x) \lambda_0(\D x) + \int_0^\infty \left(\int_0^t \E^{-(t-s)x} y(x) \,\D X_s\right)\,\nu(\D x).
\end{align*}
Applying the stochastic Fubini theorem, the total mass~$\langle 1, \lambda \rangle$ thus solves the affine SVE~\eqref{eq:SVE} in dimension one, with~$K(t)=\int_0^\infty \E^{-tx}\,\nu(\D x)$ and $\lambda_0$ being the Dirac mass at zero multiplied by~$V_0$.

This setting is not limited to the univariate case and, recalling our initial motivation, we observe that the rough Heston model implies the lift
\begin{align*}
\left\{
\begin{array}{rl}
    \D \lambda_t^1(\D x) &=
    \displaystyle -x \lambda_t^1(\D x) \dt + \delta_0(\D x) \left(-\half\langle 1, \lambda_t^2\rangle \dt + \sqrt{\langle 1, \lambda_t^2\rangle} \D W_t^1 \right), \\
    \D \lambda_t^2(\D x) &=
    \displaystyle -x \lambda_t^2(\D x) \dt + \nu(\D x) \left(\beta(\theta-\langle 1, \lambda_t^2\rangle) \dt + \sigma \sqrt{\langle 1, \lambda_t^2\rangle} \D W_t^2 \right),
\end{array}
\right.
\end{align*}
where~$W^1$ and~$W^2$ are correlated Brownian motions.
In this framework,
$\langle 1, \lambda_t^1\rangle$ represents the log asset price and
$\langle 1, \lambda_t^2\rangle$ the instantaneous variance.

The added value of this approach lies in the generalised Feller property, a notion introduced in~\cite{DT10}, which is satisfied by the solution to~\eqref{5eq:lift} (albeit with~$\theta=0$) and opens the gates to many parallels with Markovian ergodic theory. Indeed, it is an extension of the standard Feller property to spaces that are not locally compact. To the best of our knowledge, the present paper is the first to study the large-time behaviour
of affine Volterra processes through generalised Feller processes. Friesen and Karbach, in a paper posted at the same time as ours~\cite{friesen2024stationary}, exploit generalised Feller processes to study the stationarity of a class of Hilbert-valued models.
Interestingly enough, in~\cite{CT18}, the authors also prove weak existence and uniqueness of the mild solution of~\eqref{5eq:lift} by computing its Laplace transform
\[
\bE[\exp(\langle y_0,\lambda_t \rangle)] = \exp(\langle y_t, \lambda_0\rangle),
\]
where~$y$ is the unique solution of a non-linear PDE.

The ergodic theory of space-time SPDEs is classical by now, with the monographs~\cite{DPZ92,DPZ96}, and in particular the vast literature on the 2D stochastic Navier-Stokes equation which was the stage of several profound advances such as the asymptotic coupling technique and the asymptotic Feller property~\cite{FM95,Hairer05,HM06,WMS01,WM01}. More recent results include~\cite{BKS20,CKNP19,KKMS20}.

As one can expect, the literature on \textcolor{black}{measure-valued stochastic evolution equations} is less developed. However, there exist connections with measure-valued branching processes (aka superprocesses) which are also Markovian processes characterised by their Laplace transform, only the PDE satisfied by~$y$ takes a different form~\cite{Li11}. We note that the ergodic behaviour of superprocesses has been studied extensively thanks to their Laplace transform~\cite{CRY17,Etheridge93,Frisen19,Iscoe86,KLM19}, which is a hopeful message for us. Furthermore, the analogy with our \textcolor{black}{stochastic evolution equation} lift does not stop there: the density field of some of these superprocesses also satisfies an \textcolor{black}{stochastic evolution equation} with an affine structure~\cite{KS88,LWX04}, and in the continuous-state case branching processes share properties with affine SDEs~\cite{KLM12}. However, superprocesses take values in spaces of non-negative measures, which can be locally compact, in which case it makes sense to use the standard Feller property.

The rest of the paper is organised as follows. In Section~\ref{S:MainResults} we discuss the mathematical framework, the generalised Feller property, and present our main results,  Proposition~\ref{prop:existencecondition} and Theorem~\ref{thm:main}. In Section~\ref{sec:condition} we introduce a condition for existence of an invariant measure in an abstract setting that then leads to the proof of Proposition~\ref{prop:existencecondition}. In Section~\ref{S:MultiD_SPDE}, we consider the multidimensional \textcolor{black}{stochastic evolution equation}, prove that there is a unique solution to the multidimensional lift that is a generalised Feller process and further discuss its properties in Proposition~\ref{prop:GFP}. Section~\ref{sec:bound} concludes the proof of Theorem~\ref{thm:main} by showing that under Assumption~\ref{assu:main}, the bound in Proposition~\ref{prop:existencecondition} holds. Section~\ref{S:Outlook} discusses future directions.

\section{Framework and main results}\label{S:MainResults}

We recall some of the notations introduced in~\cite{CT18}.
Let~$\bm X$ be a completely regular Hausdorff topological space.
We say that $\varrho\colon \bm X\to (0,\infty)$ is an admissible weight function if the sets~$K_R:= \{x\in \bm X: \varrho(x)\le R\}$ are compact for all~$R>0$~\cite[Definition 2.1]{CT18}. The supremum norm on~$\bR$ is denoted~$\norminf{\cdot}$. The vector space
\[
\mathrm{B}^\varrho(\bm X) := \left\{ f\colon \bm X\to \bR: \sup_{x\in \bm X}\varrho(x)^{-1} \norminf{f(x)} <\infty \right\},
\]
equipped with the norm
\begin{equation}\label{eq:rhonorm}
\norm{f}_\varrho:= \sup_{x\in \bm X} \varrho(x)^{-1}\norminf{f(x)},
\end{equation}
is a Banach space, and~$\cC_b(\bm X) \subset \mathrm{B}^\varrho(\bm X)$. Moreover, the space~$\cB^\varrho (\bm X)$ is defined as the closure of~$\cC_b(\bm X)$ in~$\mathrm{B}^\varrho(\bm X)$~\cite[Definition 2.3]{CT18}, and is also a Banach space when equipped with the norm~\eqref{eq:rhonorm}.
Generalised Feller semigroups are the analogue of standard Feller semigroups on the space of
continuous functions vanishing at infinity on locally compact spaces. They are bounded, positive, linear, strongly continuous operators.
\begin{definition}[Definition 2.5 of~\cite{CT18}]
A family of bounded linear operator~$P_t:\cB^\varrho(\bm X)\to\cB^\varrho(\bm X)$ for~$t\ge0$ is called generalised Feller semigroup if
\begin{enumerate}[(i)]
    \item $P_0=I$, the identity on~$\cB^\varrho(\bm X)$,
    \item $P_{t+s}= P_t P_s$, for all~$t,s\ge0$,
    \item For all~$f\in\cB^\varrho(\bm X)$ and~$x\in\bm X$, $\lim_{t\downarrow0} P_t f(x) = f(x)$,
    \item There exist~$C>0$ and~$\ep>0$ such that for all~$t\in[0,\ep],$ $\norm{P_t}_{L(\cB^\varrho(\bm X))}\le C$,
    \item $P_t$ is positive for all~$t\ge0$, that is, for any $f\in\cB^\varrho(\bm X)$ such that $f\ge0$, then $P_t f\ge0$.
\end{enumerate}
\end{definition}
Theorems 2.11 and 2.13 of  \cite{CT18} ensure that each generalised Feller semigroup gives rise to an associated Markov process and such process has a version with c\`agl\`ad paths.

Let~$\cP(\bm X)$ be the space of probability measures on~$\bm X$.
\begin{definition}
For all~$t\ge0$, we define~$P^\ast_t$ as the adjoint of~$P_t$, that is, for all~$\vphi\in\cB^\varrho(\bm X)$ and~$\mu\in\cP(\bm X)$, they satisfy
\[
\int_{\bm X} P_t \vphi(x) \mu(\D x) =: \langle P_t \vphi, \mu\rangle = \langle \vphi, P^\ast_t \mu\rangle.
\]
\end{definition}
We will call $(\lambda_t)_{t\ge0}$ the process associated to the semigroup~$(P_t)_{t\ge0}$ (and reciprocally) if, for all $t\ge0$, $P_t^\star \gamma$ is the law of $\lambda_t$ whenever $\gamma$ is the law of $\lambda_0$. From the display above, one deduces that in that case~$P_t\vphi(\lambda_0)=\bE_\gamma[\vphi(\lambda_t)]$.
\begin{definition}
We say that~$\mu$ is an invariant measure if $\langle P_t \vphi, \mu\rangle=\langle \vphi, \mu\rangle$ for all~$t\ge0$ and for all~$\vphi\in\cB^\varrho(\bm X)$.
\end{definition}

\subsection{Main results}
We start with a condition for existence in a general state-space $Y^\star$, defined as the dual of a Banach space~$Y$ and equipped with its weak-$\star$-topology. We also define the strong norm~$\norm{\lambda}_{Y^\star} := \sup_{y\in Y,\,\norm{y}\le1}\langle y,\lambda \rangle$.
\begin{proposition}\label{prop:existencecondition}
If~$(\lambda_t)_{t\ge0}$ is a generalised Feller process taking values in $Y^\star$ and  \vspace{-5pt}
\begin{equation}\label{eq:conditionnorm}
    \sup_{t\ge 0} \bE[\norm{\lambda_t}_{Y^\star}] <\infty,
    \vspace{-5pt}
\end{equation}
then it has an invariant measure.
\end{proposition}
\begin{remark}
Equation~\eqref{eq:conditionnorm} corresponds to $\sup_{t\ge 0} \bE[\varrho(\lambda_t)] <\infty$ for the choice of weight function~$\varrho(\lambda):=1+\norm{\lambda}_{Y^\star}$.
\end{remark}

Our goal is to apply this to $Y^\star$-valued \textcolor{black}{stochastic evolution equation}s which we introduce now.
\textcolor{black}{Let~$d\ge1$, we consider the Banach space $Y:= \cC_b(\overline{\bR}_+,\bR^d)$, where~$\overline{\bR}_+:=\bR\cup{+\infty}$ (compactifying $\bR_+$ makes $\cC_b(\overline{\bR}_+,\bR^d)$ separable) and its dual~$Y^\star=\cM(\overline{\bR}_+,\bR^{d})$ is the space of signed measured, as explained in~\cite[Chapter IV]{DufordSchwartz}.}
We then have~$\langle y,\lambda\rangle := \sum_{i=1}^d \langle y^i ,\lambda^i\rangle_i:= \sum_{i=1}^d  \int_{0}^\infty y^i(x) \lambda^i(\D x)$ for all~$y\in Y$, $\lambda\in Y^\star$.
Moreover, the weight function is~$\varrho(\lambda) := 1+ \norm{\lambda}_{Y^\star}$, the set~$\bm X$ introduced in \cite[Section~1.3]{CT18}
is a subset of~$Y^\star$ and the space~$\cB^\varrho(\bm X)$ is defined in the same way.

We consider the $Y^\star$-valued multidimensional \textcolor{black}{measure-valued stochastic evolution equation}
\begin{equation}\label{eq:multiSPDE}
\D \lambda_t(\D x) = -x \lambda_t(\D x) \dt + \nu(\D x) \,\D X_t,
\end{equation}
where, denoting $\bbl_s := \langle 1,\lambda_s \rangle$, ~$X$ is the~$\bR^d$-semimartingale:
\begin{equation}\label{eq:XProcess}
    X_t := \int_0^t  b(\bbl_s) \ds + \int_0^t\sigma(\bbl_s) \D W_s,
\end{equation}
with $\bbl:=(\olambda^1,\cdots,\olambda^d)$,
~$\lambda_t(\D x)\in\bR^d$, $\nu(\D x)\in\bR^{d\times d}$, $W$ is an $m$-dimensional Brownian motion,
and $b\in\bR^d$ and $\sigma\in\bR^{d\times m}$ are such that
\begin{equation}\label{eq:multicoefs}
b_i(x)=-\sum_{k=1}^d \beta_{ik} x_k;
\quad \sigma_{ij}(x)=\sum_{k=1}^d \sigma_{ijk} \sqrt{x_k}+c_{ijk}, \quad\text{for all } x = (x^1, \cdots, x^d)\in\bR^d.
\end{equation}
Let us define~$\tilde{d}\in\llbracket 0,d\rrbracket$ and the state space~$D:=\bR_+^{\tilde{d}}\times \bR^{d-\tilde{d}}$.
\begin{assumption}\label{assu:main}
The following conditions hold:
\begin{enumerate}[a)]
    \item $\lambda_0(\D x) = V_0 \delta_0(\D x)$ for some~$V_0\in D$;
    \item for all $i\in \llbracket 1,d\rrbracket$, $\nu^i(\D x)= \frac{1}{\Gamma(\alpha(i))\Gamma(1-\alpha(i))}(x-\delta)^{-\alpha(i)} \one_{x>\delta}\dx$, where $\alpha(i)\in(\half,1]$ and $\delta>0$;
    \item $c_{ijk}=0$ for all $i,j,k$;
    \item if~$i\le\tilde{d}$ and $i\neq k$, then $\sigma_{ijk}=\beta_{ik}=0$ for all $j$;
    \item if $i>\tilde{d}$ and $k>\tilde{d}$, then $\sigma_{ijk}=\beta_{ik}=0$ for all $j$.
    \item \textcolor{black}{$\beta_{ii}>0$ for $i=1,\cdots,\tilde{d}$  and $\beta_{jk}\leq 0$ for $j,k=1,\cdots,\tilde{d}$ and $j\neq k$.}
\end{enumerate}
\end{assumption}
Our main result is the following theorem.
\begin{theorem}\label{thm:main}
If Assumption~\ref{assu:main} holds,
then the solution to~\eqref{eq:multiSPDE} has an invariant measure.
\end{theorem}
\begin{remark}
Assumption~\ref{assu:main} has specific structural restrictions on the coefficients $\beta,\sigma$ under which Theorem~\ref{thm:main} holds. We note here though that these are only sufficient conditions and in fact, intermediate results of independent interest, such as Proposition~\ref{prop:GFP} and Lemma~\ref{lemma:conditiondX}, do not require these conditions. \textcolor{black}{ We also note that even though our main result holds with $c_{ijk}=0$ for all $i,j,k$ in \eqref{eq:multicoefs}, certain results in the paper are proven with non-zero $c_{ijk}$, see Section \ref{S:MultiD_SPDE}.}
\end{remark}

Condition d) states that for $i\le \tilde{d}$ the total mass~$\olambda^i$ of the $i$th component is an autonomous square-root process. The last condition entails that for $i>\tilde{d}$ the $i$th component is explicitly given in terms of the autonomous square-root processes.
In the mild sense, \eqref{eq:multiSPDE} reads
\begin{equation}\label{eq:mildSPDE}
    \langle y, \lambda_t\rangle = \langle \E^{-t\cdot}y, \lambda_0 \rangle + \int_0^\infty y(x) \left(\int_0^t \E^{-x(t-s)}\,\D X_s\right) \nu(\D x),
\end{equation}
for any~$y\in Y$ and $\lambda_0\in Y^\star$, and for all~$i\in\llbracket 1,d\rrbracket$,
\[
\langle y^i, \lambda_t^i\rangle = \langle \E^{-t\cdot}y^i, \lambda_0^i \rangle + \int_0^\infty y^i(x) \left(\int_0^t \E^{-x(t-s)}\,\D X^i_s\right) \nu^i(\D x).
\]
Assumption~\ref{assu:main} then implies
\begin{alignat*}{2}
\D X^i_t &=  -\beta_{ii}\olambda^i_t\dt +  \sqrt{\olambda^i_t} \sum_{j=1}^m \sigma_{iji}\D W^j_t, \qquad\qquad\qquad && \text{if  } i\in\llbracket 1,\tilde{d}\rrbracket,\\
\D X^i_t &= -\sum_{k=1}^{\tilde{d}} \beta_{ik} \olambda^k_t \dt + \sum_{j=1}^m \sum_{k=1}^{\tilde{d}} \sigma_{ijk} \sqrt{\olambda_t^k} \,\D W^j_t, && \text{if  } i\in\llbracket\tilde{d}+1,d\rrbracket.
\end{alignat*}

The representation~\eqref{eq:mildSPDE} allows to derive the equation satisfied by~$\olambda$ in certain cases of interest.
\begin{example}\label{examples} \
\begin{itemize}
    \item The one-dimensional Volterra square-root process, appearing as the variance in the rough Heston model~\cite{EFR18}
    (without long-term mean) satisfies these conditions with~$\tilde{d}=d=1$ and~$\alpha(i)=\alpha\in(\half,1)$:
    \[
    \olambda_t = V_0 + \int_0^t (t-s)^{\alpha-1}\left(-\beta \olambda_s\right) \ds + \int_0^t (t-s)^{\alpha-1}\sigma \sqrt{\olambda_s}\,\D W_s .
    \]
    Its lift is the one-dimensional \textcolor{black}{stochastic evolution equation} introduced in~\cite{CT18}.
    \item The two-dimensional rough Heston model (without long-term mean)
    is also covered where~$\tilde{d}=1$, $d=m=2$, $\alpha(1)=\alpha\in(\half,1)$ and $\alpha(2)=1$:
    \begin{equation*}
        \left\{
        \begin{array}{rl}
        \olambda_t^{1} &=\displaystyle V_0^1 + \int_0^t (t-s)^{\alpha-1}\left(-\beta \olambda_s^{1}\right) \ds + \int_0^t (t-s)^{\alpha-1}\sigma \sqrt{\olambda_s^{1}}\,\D W_s^{1} \\
        \olambda_t^{2} &=\displaystyle -\half\int_0^t \olambda_s^{1}\ds + \int_0^t \rho \sqrt{\olambda_s^{1}}\,\D W^{1}_s + \int_0^t \bar\rho  \sqrt{\olambda_s^{1}}\,\D W^{2}_s.
        \end{array}
        \right.
    \end{equation*}
    More precisely, the coefficients are $\sigma_{ij1}=0$ for all~$i,j=1,2$ and
\[
\bm{\beta} = \begin{pmatrix}0& 1/2 \\ 0& \beta \end{pmatrix}; \quad
\sigma_{\cdot\cdot2} = \begin{pmatrix} \overline\rho & \rho  \\ 0 & \sigma \end{pmatrix}.
\]
    \item One can also consider extensions of the previous example with $\alpha(2)<1$, and higher dimensional systems where $\tilde{d}>1$ and the square-root processes feed back into the dynamics of~$\olambda_t^d$, i.e.~$\beta_{dk}>0$ for all~$k\le \tilde{d}$.
\end{itemize}
\end{example}

\begin{proof}[Proof of Theorem~\ref{thm:main}]
This is a combination of Propositions~\ref{prop:existencecondition},~\ref{prop:GFP} and~\ref{prop:bound}.
More precisely, Proposition~\ref{prop:existencecondition} states that, for a generalised Feller process~$(\lambda_t)_{t\ge0}$, the bound~\eqref{eq:conditionnorm} is a sufficient condition for the existence of an invariant measure. Proposition~\ref{prop:GFP} shows that there exists a unique solution~$(\lambda_t)_{t\ge0}$ to~\eqref{5eq:lift} and that it is indeed a generalised Feller process. Finally, Proposition~\ref{prop:bound} ensures the condition holds.
\end{proof}
Theorem~\ref{thm:main} delivers a new notion of stationarity for the Volterra process~$(\olambda_t)_{t\ge0}$. Straightforward computations show that, for any $s<t$,
\begin{equation}\label{eq:SVE1}
\olambda_t = \int_0^\infty \E^{-x(t-s)} \lambda_s(\D x) + \int_s^t K(t-r)\D X_r,
\end{equation}
where~$K(t)=\int_0^\infty \E^{-xt}\,\nu(\D x)$ and $X$ solves~\eqref{eq:XProcess}.
The choice of~$\nu$ in Assumption~\ref{assu:main}
yields the kernel~$K^i(t)=\frac{t^{\alpha(i)-1}}{\Gamma(\alpha(i))}\E^{-\delta t}$,
and weak existence and uniqueness were derived in~\cite{ALP17}.
 Equation~\eqref{eq:SVE1} emphasises that, as~$\olambda$ is an evaluation of~$\lambda$ on a lower-dimensional space, the law of~$\olambda_t$ conditional to $\cF_s:=\sigma(\{W_r: r\le s\})$ with $s<t$ is equal to the law of $\olambda_t$ conditional to $\lambda_s$.
\begin{corollary}\label{coro:statV}
\textcolor{black}{Under Assumption \ref{assu:main}, there exists a probability measure~$\mu^\star$ on~$Y^\star=\cM(\overline{\bR}_+,\bR^{d})$ such that if, for any $s\geq 0$, $\mu^\star$ is the distribution of~$\lambda_s$
then~$(\olambda_t)_{t\ge s}$ is a strictly stationary process on~$\bR_+$.}
\end{corollary}
\begin{proof}
Let~$\mu^\star$ denote an invariant distribution of~$\lambda$, which we note is constructed as a limiting distribution. Provided $\lambda_s$ follows this distribution, the process $(\lambda_t)_{t\ge s}$ is strictly stationary and the relation $\olambda=\langle 1,\lambda \rangle$ then immediately implies the claim.
\end{proof}
It is insightful to compare this type of stationarity for fractional process with existing results.
\textcolor{black}{
\begin{itemize}
    \item In the closest paper to ours \cite{FJ22}, the authors prove that $(\olambda_{t+h})_{t\ge 0}$ converges in law towards a stationary process~$(\olambda_{t}^{{\rm stat}})_{t\ge 0}$ as $h\to\infty$. They characterise properties of the marginal law of this limiting process but do not provide information on its dynamics, in particular it is not known whether~$\olambda^{{\rm stat}}$ satisfies an SVE. Moreover, this convergence result does not imply that $\olambda$ itself is a stationary process.
    \item Hairer \cite{Hairer05} on the other hand looks for an invariant measure on the product space of~$\bR$ (where the fractional OU process lives) and the past of the driving noise (an infinite-dimensional space).
    \item Fractional SDEs with additive noise are also studied by~\cite{li2022non}:
    they find stationary solutions~$Y$ on $\cC(\bR_+)$, meaning  that $(Y_t)_{t\ge0} = (Y_{t+h})_{t\ge0}$ in distribution for all~$h\ge0$.
\end{itemize}
In contrast, our approach allows to construct a stationary process directly on~$\bR_+$.}

\notthis{
\begin{remark}
The authors in~\cite{CT18} provide two types of lifts for the SVE~\eqref{eq:SVE}. The forward curve lift (Section 5.2 of~\cite{CT18}) lies in the Filipovic space which dual is itself. Weak-$\star$-compactness is more affordable than in the standard space of continuous functions. We leave the study of this lift for a subsquequent paper.
The other type of lift,  introduced in Section 5.1 of~\cite{CT18} and displayed in~\eqref{5eq:lift}, takes value in the space of signed measures on~$\overline\bR_+$, which nice tightness criterion is exploited in this paper.
\end{remark}
}

We close this section with an important remark on the consequences of Theorem~\ref{thm:main} and of Corollary~\ref{coro:statV}. Our goal here is to show that in the case of the kernel~$K^i(t)=\frac{t^{\alpha(i)-1}}{\Gamma(\alpha(i))}\E^{-\delta t}$ with $\delta>0$, the limiting distribution of $ V_t$ in \eqref{eq:SVE} is not that of the zero process unless $V_{0}=0$.

Let us denote $V_t^k=\olambda^k_t$. Then, we have
\begin{equation}\label{eq:ExpectationV_tfixed}
    \bE[V_t^k]= V_0^k\left( 1-  \beta_{kk}
    \int_{0}^{t}\E^{-\delta s}s^{\alpha-1} \mathrm{E}_{\alpha,\alpha}(-\beta_{kk} s^\alpha)\D s\right).
\end{equation}

Based on integral representations for the Mittag-Leffler function~\cite[Section~7]{Haubold2011}, we obtain
\begin{equation}\label{eq:ExpectationV_tinfinity}
\lim_{t\uparrow\infty}\bE[V_t^k] = V_0^k\left( 1-  \beta_{kk}
    \int_{0}^{\infty}\E^{-\delta s}s^{\alpha-1} \mathrm{E}_{\alpha,\alpha}(-\beta_{kk} s^\alpha)\D s\right)
    = V_0^k\left( 1-  \frac{\beta_{kk} }{\beta_{kk}+\delta^{\alpha}}
    \right).
\end{equation}

This calculation shows that
\begin{itemize}
\item{If $\delta=0$, then $\lim_{t\uparrow\infty}\bE[V_t^k]=0$ no matter what the value of $V_0^k$ is. It is indeed true that if $\delta=0$, then $V_t$ converges to the trivial zero process.}
\item{If $\delta>0$, then $\lim_{t\uparrow\infty}\bE[V_t^k]=0$ if and only if  $V_0^k=0$.}
\end{itemize}

So if $\delta\neq 0$ and $V_0^{k}\neq 0$, then $\lim_{t\uparrow\infty}\bE[V_t^k]\neq 0$. It is also interesting to note that in this case the limiting behaviour depends on the initial condition $V_0^{k}\neq 0$, a point also noted in~\cite{FJ22}.  Hence, in the case of $\delta\neq 0$ and $V_0^{k}\neq 0$ the limiting behaviour of $V_t$ is not zero, whereas if $\delta=0$, then no matter what the initial condition $V_0^{k}$ is, we will have that $\lim_{t\uparrow\infty}\bE[V_t^k]=0$.

\color{black}

\section{A condition for existence}\label{sec:condition}
This section aims at proving Proposition~\ref{prop:existencecondition}.
We start by stating and proving an extension of Krylov-Bogoliubov theorem~\cite[Theorem~11.7]{DPZ92} to the setting of generalised Feller processes on~$\cB^\varrho(\bm X)$ for any completely regular Hausdorff topological space~$\bm X$.
\begin{lemma}\label{lemma:Krylov}
Let~$(P_t)_{t\ge0}$ be a generalised Feller semigroup.
Suppose that there exists~$\gamma \in \cP(\bm X)$ and a strictly positive sequence~$T_n$ going to~$+\infty$ as~$n$ goes to~$+\infty$ such that \vspace{-5pt}
\begin{enumerate}[(i)]
    \item the sequence of measures $Q^\ast_{T_n}\gamma:=\frac{1}{T_n} \int_0^{T_n} P^\ast_t\gamma\dt$ converges weakly to some~$\mu\in \cP(\bm X)$;
    \item It holds that
    \begin{align}\label{eq:AbstractBound}
        \sup_{t\ge0}\int_{\bm{X}} P_{t} \varrho(\lambda)\gamma(\D\lambda)<\infty
    \end{align}
\end{enumerate}
Then~$\mu$ is an invariant measure for~$(P_t)_{t \ge0}$.
\end{lemma}
\begin{remark}
The original Krylov-Bogoliubov theorem for standard Feller semigroups does not require~(ii). 
\end{remark}
On Polish spaces,
Prokhorov's theorem~\cite[Theorem~2.3]{DPZ92} states that (i) is equivalent to tightness of~$(P_t^\ast \gamma)_{t\ge0}$.
We will invoke an extension of this result to completely regular spaces~\cite[Chapitre IX, Section 5.5, Th\'eor\`eme 1]{bourbaki2006mesures} as dual spaces equipped with the weak-$\star$-topology are not Polish. This result relies upon the so-called Prokhorov condition for a subset~$H$ of the space of bounded Radon measures~$\cM_b(T)$, where $T$ is completely regular. In the case~$H=\cP(\bm{X})\subset \cM_b(\bm{X})$ of interest to us, this  preliminary condition is equivalent to tightness. We now state the theorem.
\begin{theorem}[Bourbaki]\label{th:Prokhorov}
    Let $T$ be a completely regular space, and $H$ a subset of~$\cM_b(T)$ satisfying Prokhorov's condition; then $H$ is relatively compact in $\cM_b(T)$ equipped with the topology of weak convergence.
\end{theorem}
As the dual of a separable Banach space equipped with the weak-$\ast$-topology, $Y^\star$ is a completely regular space and so is the subset~$\bm{X}$.
\begin{proof}
Fix~$t > 0$ and~$\vphi \in\cB^\varrho(\bm X)$, then by definition~$P_t \vphi \in \cB^\varrho(\bm X)$.
Let~$R>0$, we know from~\cite[Equation (2.3)]{CT18} that $f\in\cB^\varrho(\bm X)$ if and only if~$f\lvert_{K_R} \in \cC(K_R)$, \textcolor{black}{where we recall~$K_R= \{x\in \bm X: \varrho(x)\le R\}$}. We thus need to work on~$K_R$ to apply the weak convergence of (i):
\[
\langle \vphi, P^\ast_t \mu\rangle
= \langle P_t \vphi, \mu\rangle
= \int_{\bm X \setminus K_R} P_t \vphi \,\D\mu + \int_{K_R} P_t \vphi \,\D\mu
= \int_{\bm X \setminus K_R} P_t \vphi \,\D\mu + \lim_{n\uparrow\infty} \int_{K_R} P_t \vphi \,\D Q_{T_n}^\ast \gamma.
\]
\textcolor{black}{Note that in the last computation, weak convergence of $Q_{T_n}^\ast \gamma$ to $\mu$ can be used even though the function $\one_{K_R}\vphi$ is not in~$\cC(\bm X)$.
This can be done following the proof of approximation by appropriate cut-off functions as in \cite[Theorem 29.1]{Billingsey95}.}
Then we go back to $\bm X$ to apply the adjoint property
\begin{align}
\int_{K_R} P_t \vphi \,\D Q_{T_n}^\ast \gamma
&= \frac{1}{T_n} \int_0^{T_n}\langle P_t \vphi , P_{s}^\ast \gamma \rangle \ds - \int_{\bm X\setminus K_R} P_t \vphi \,\D Q_{T_n}^\ast \gamma\nonumber\\
&=\frac{1}{T_n} \int_t^{t+T_n} \langle \vphi , P_{s}^\ast \gamma \rangle\ds - \int_{\bm X\setminus K_R} P_t \vphi \,\D Q_{T_n}^\ast \gamma.\nonumber
\end{align}
We observe that
\[
\frac{1}{T_n} \int_t^{t+T_n} \langle \vphi , P_{s}^\ast \gamma \rangle\ds
=  \frac{1}{T_n} \int_0^{T_n} \langle \vphi , P_{s}^\ast \gamma \rangle\ds
+ \frac{1}{T_n} \int_{T_n}^{t+T_n} \langle \vphi , P_{s}^\ast \gamma \rangle\ds
- \frac{1}{T_n} \int_0^{t} \langle \vphi , P_{s}^\ast \gamma \rangle\ds,
\]
where the second and third terms tend to zero as $n$ goes to infinity, while the first one is~$\langle\vphi,Q^\ast_{T_n}\gamma\rangle$. We relocate to $K_R$ for the weak convergence:
\begin{align}
\lim_{n\uparrow\infty} \frac{1}{T_n}\int_t^{t+T_n} \langle \vphi , P_{s}^\ast \gamma \rangle \ds &= \lim_{n\uparrow\infty} \left( \int_{K_R} \vphi \,\D Q_{T_n}^\ast \gamma + \int_{\bm X \setminus K_R} \vphi \,\D Q_{T_n}^\ast \gamma  \right) \nonumber\\
&= \lim_{n\uparrow\infty} \left( \int_{\bm X} \one_{K_R}\vphi \,\D Q_{T_n}^\ast \gamma + \int_{\bm X \setminus K_R} \vphi \,\D Q_{T_n}^\ast \gamma  \right) \nonumber\\
&= \int_{K_R} \vphi \,\D\mu + \lim_{n\uparrow\infty} \int_{\bm X \setminus K_R} \vphi \,\D Q_{T_n}^\ast \gamma.\nonumber
\end{align}
In the last computation, once again the weak convergence of $Q_{T_n}^\ast \gamma$ to $\mu$ can be used even though the function $\one_{K_R}\vphi$ is not in~$\cC(\bm X)$, as in~\cite[Theorem 29.1]{Billingsey95}. Overall, this yields our objective plus a remainder
\begin{align*}
\langle \vphi, P^\ast_t \mu\rangle &= \langle \vphi,\mu\rangle - \int_{\bm X\setminus K_R} \vphi \,\D\mu + \int_{\bm X \setminus K_R} P_t \vphi \,\D\mu + \lim_{n\uparrow+\infty} \left(\int_{\bm X \setminus K_R} \vphi \,\D Q_{T_n}^\ast \gamma - \int_{\bm X\setminus K_R} P_t \vphi \,\D Q_{T_n}^\ast \gamma \right)\\
&=:  \langle \vphi,\mu\rangle + \ep(R).
\end{align*}

To deal with $\ep(R)$, we apply a Fatou-type lemma for measures~\cite[Theorem~A.3.12]{DE97}. For any~$f\in\cB^\varrho(\bm X)$,
\begin{align}
\Big\lvert \int_{\bm X\setminus K_R} f \,\D\mu \Big\lvert
\le \sup_{x\in\bm X\setminus K_R} \frac{\abs{f(x)}}{\varrho(x)} \int_{\bm X} \varrho \,\D \mu
&\le \sup_{x\in\bm X\setminus K_R} \frac{\abs{f(x)}}{\varrho(x)}  \liminf_{n\uparrow+\infty} \int_{\bm X} \varrho \,\D Q_{T_n}^\ast \gamma \nonumber\\
&= \sup_{x\in\bm X\setminus K_R} \frac{\abs{f(x)}}{\varrho(x)}  \liminf_{n\uparrow+\infty} \frac{1}{T_n}\int_0^{T_n} P_s \varrho(\lambda)\ds , \label{eq:EstimateFatou}
\end{align}
where~$\lambda$ is a $\bm X$-valued random variable with distribution~$\gamma$.
By integrating over $\gamma$ both sides of~\eqref{eq:EstimateFatou} and noting this integral has mass one we get
$$
\Big\lvert \int_{\bm X\setminus K_R} f \,\D\mu \Big\lvert \le \sup_{x\in\bm X\setminus K_R} \frac{\abs{f(x)}}{\varrho(x)}  \left( \sup_{t\ge0}\int_{\bm X} P_{t} \varrho(\lambda)\gamma(\D\lambda)\right)
$$
The same computations also yield the same estimate for
$\abs{\int_{\bm X \setminus K_R} f \,\D Q_{T_n}^\ast \gamma} $.
By~\cite[Equation~(2.3)]{CT18} we have for any~$f\in\cB^\varrho(\bm X)$,
\[
\lim_{R\uparrow\infty} \sup_{\bm X\setminus K_R} \frac{\abs{f(x)}}{\varrho(x)} =0.
\]
The uniform bound~\eqref{eq:AbstractBound} thus ensures that~$\ep(R)$ vanishes as~$R$ goes to $+\infty$.
\end{proof}
\begin{proof}[Proof of Proposition~\ref{prop:existencecondition}]
Let~$(P_t)_{t\ge0}$ be the generalised Feller semigroup associated to~$(\lambda_t)_{t\ge0}$. Hence, for any~$\bm X$-valued random variable~$\lambda_0$, $P_t\varrho(\lambda_{0})=\bE_{\lambda_0\sim\gamma}[\varrho(\lambda_t)]$.

In~\cite{CT18}, the analysis is performed with the weight function $\varrho(\lambda) = 1+\norm{\lambda}^2_{Y^\star}.$
The only necessary estimate is~$\bE[\varrho(\lambda_t)]\le C \varrho(\lambda_0)$, for some finite constant~$C>0$, and we note that the inequality~$\bE[1+\norm{\lambda_t}^2_{Y^\star}] \le C (1+\norm{\lambda_0}^2_{Y^\star})$ implies~$\bE[1+\norm{\lambda_t}_{Y^\star}] \le \sqrt{2C} (1+\norm{\lambda_0}_{Y^\star})$.
Therefore, the weight function~$\varrho(\lambda) =1+\norm{\lambda}_{Y^\star}$ is also admissible, and the assumption~\eqref{eq:conditionnorm} implies
\begin{equation}\label{eq:conditionrho}
\sup_{t\ge 0} \int_{\bm{X}} \bE_{\lambda_0}[\varrho(\lambda_t)]  \D \gamma(\lambda_0)=\sup_{t\ge 0}\bE[\varrho(\lambda_t)] <\infty.
\end{equation}
This yields condition (ii) of Lemma~\ref{lemma:Krylov}.

For all~$t\ge0$, let $\Lambda_t$ be the $Y^\star$-valued random variable with distribution~$Q^\ast_t\gamma$. Inspection of the proof of~\cite[Lemma~2.9]{BD19} shows that, even though~$Y^\star$ is not a Polish space, the family $(\Lambda_t)_{t\geq 0}$ is tight if there exists a tightness function $G:Y^\star\to\bR_+$ such that~\mbox{$\sup_{t\ge0} \bE[G(\Lambda_t)]<\infty$.} We observe that an admissible weight function is also a tightness function by design, hence proving tightness of~$(\Lambda_t)_{t\ge 0}$ boils down to showing
\[
\sup_{T\ge 0} \frac{1}{T}\int_0^T \bE[\varrho(\lambda_t)] \dt <\infty,
\]
which is a consequence of~\eqref{eq:conditionrho}. We then invoke the version of Prokhorov's theorem on completely regular spaces (Theorem~\ref{th:Prokhorov}) to deduce that~$(Q^\ast_t\gamma)_{t\ge0}$ is relatively compact in $\cP(\bm{X})$ and hence has a converging subsequence.

Therefore, both conditions of Lemma~\ref{lemma:Krylov} are satisfied as soon as~\eqref{eq:conditionnorm} holds.
\end{proof}

\section{The multidimensional \textcolor{black}{measure-valued stochastic evolution equation}}\label{S:MultiD_SPDE}

For simplicity, we assume that~$\nu$ is diagonal, which releases some indexation load and appoints the same kernel for drift and diffusion.
We can then define the matrix~$K(t):=\langle \E^{-t\cdot} ,\nu\rangle$ and since~$\nu$ is diagonal this yields~$K^i(t)=\langle \E^{-t\cdot} ,\nu^i\rangle_i$ where the superscript~$i$ stands for the~$i$th component of the diagonal.

We replicate the setting of Markovian affine processes characterised in~\cite{DFS03} and cast it in the appropriate multidimensional adaptation of~\cite{CT18}.
In the setting of~\eqref{eq:multiSPDE}, we want the total mass~$\olambda^i$ of each component to be either a square root process or an Ornstein-Uhlenbeck-type process; this is achieved by the following set of assumptions:
\begin{assumption}\label{assu:affine}
The following hold:
\vspace{-.3cm}
\begin{itemize}
    \item if~$i\le\tilde{d}$ then~$c_{ijk}=0$ for all $j,k$; if moreover $k\neq i$ then $\sigma_{ijk}=\beta_{ik}=0$ for all $j$;
    \item if $i>\tilde{d}$ and $k>\tilde{d}$, then $\sigma_{ijk}=0$ for all $j$.
\end{itemize}
\end{assumption}
\begin{remark}
This assumption is strictly weaker than Assumption~\ref{assu:main}, hence Examples~\ref{examples} are still covered.
\end{remark}
We can now display the equations satisfied by the total mass of each component.
For each~$i\in\llbracket 1,\tilde{d}\rrbracket$, $\overline{\lambda}^i$ is then an autonomous square-root process living in~$\bR_+$, and satisfies
\begin{equation}\label{eq:rootcomponent}
\olambda^i_t = \langle \E^{-t\cdot}, \lambda_0^i\rangle + \int_0^t K^i(t-s) (-\beta_i\olambda^i_s)\ds +  \int_0^t K^i(t-s)\sqrt{\olambda^i_s} \sum_{j=1}^m \sigma_{ij}\D W^j_s,
\end{equation}
with no feedback from the other components.
For each~$i\in\llbracket \tilde{d}+1,d\rrbracket$, $\olambda^i$~is an OU-type process living in~$\bR$, which allows feedback from every component in the drift but only from the first~$\tilde{d}$ square-root components in the diffusion:
\begin{equation}\label{eq:OUcomponent}
\olambda^i_t = \langle \E^{-t\cdot}, \lambda_0^i\rangle + \int_0^t K^i(t-s) \left(-\sum_{k=1}^d \beta_{ik}\olambda^k_s\right)\ds + \sum_{j=1}^m \sum_{k=1}^{\tilde{d}} \int_0^t K^i(t-s)\left(\sigma_{ijk} \sqrt{\olambda^k_s}+c_{ijk}\right)\D W_s^j.
\end{equation}
The reason is that a component needs to lie on the positive cone if its square root appears somewhere in the \textcolor{black}{stochastic evolution equation}.




We now go through the one-dimensional results of~\cite[Section 4]{CT18} and highlight the differences with the multidimensional case presented here. Instead of the spaces~$\cE^w$ introduced in~\cite[(4.7)]{CT18}, we define for all~$n\in\bN$,
    $$
    \cE^n := \left\{ \lambda_0\in Y^\star : \overline{\eta}^{i,n}_t\ge0, \text{ for all }i\in\llbracket1,\tilde{d}\rrbracket, \; \eta^{i,n}_t\in\cM(\overline\bR_+,\bR) \text{ for all }i\in\llbracket \tilde{d}+1,d\rrbracket \text{ and } t\ge0\right\},
    $$
    where
\begin{equation}
    \left\{
    \begin{array}{ll}
    \displaystyle \overline{\eta}^{i,n}_t = \langle \E^{-t\cdot}, \lambda_0^i\rangle - \bigg(\beta_i-n\sum_{j=1}^m \sigma_{ij}^2\bigg) \int_0^t K^i(t-s) \overline{\eta}^{i,n}_s\ds, \qquad\qquad\qquad &i\in\llbracket1,\tilde{d}\rrbracket,\\
    \displaystyle  \D\eta^{i,n}_t = \cA^\star \eta^{i,n}_t\dt + \nu \left( \sum_{k=1}^d \beta_{ik} \overline{\eta}^{k,n}_t - n\sum_{j=1}^m \sum_{k=1}^{\tilde{d}} \left(\sigma_{ijk} \sqrt{\overline{\eta}^{k,n}_s}+c_{ijk}\right)^2 \right)\dt, \quad \eta_0^i = \lambda_0^i, & i\in\llbracket \tilde{d}+1,d\rrbracket,
    \end{array}
    \right.
    \label{eq:Equationeta}
\end{equation}
    where ~$\cA^\star$ is the linear operator of the form ~$\cA^\star \lambda(\D x)=-x\lambda(\D x)$.
Note that $\eta$ satisfies a deterministic equation and the condition for~$i\in\llbracket \tilde{d}+1,d\rrbracket$ in the definition of~$\cE^n$ is necessarily satisfied, while we want to restrict the state space of the square-root processes to $\bR_+$.
These definitions allow to define the invariant space $\displaystyle\cE:=\cap_{n\in\bN} \cE^n$ and its polar cone~$\cE_\star := \big\{ y\in Y: \langle y,\lambda\rangle \le 0 \text{ for all } \lambda\in\cE\big\}$.
For~$\cE$ to be well defined we had to let the coefficients grow with~$n$ at a uniform speed across the components, which is a notable difference with the one-dimensional framework of~\cite{CT18} where a single speed~$w$ was needed.

To define the generator of the \textcolor{black}{stochastic evolution equation}, we introduce the set~$\cD:=\{ y\in Y : \langle y,\nu \rangle \text{ is well-defined} \}$ and, for each~$y\in \cD$, the set~$\bm F$ of Fourier basis elements of the form \vspace{-5pt}
\begin{equation}\label{eq:FourierElts}
f_y: \cE\to[0,1]; \,\lambda\mapsto \exp(\langle y,\lambda \rangle).
\end{equation}
We also recall that the resolvent of the second kind corresponding to~$K$ is the kernel~$R\in L^1(\bR_+,\bR^{d\times d})$ such that
\[
K\ast R = R\ast K = K-R.
\]
The following is the analogue of~\cite[Theorem~4.17]{CT18}.
\begin{proposition}\label{prop:GFP}
Let Assumption~\ref{assu:affine} hold.
Assume moreover that~$\lambda_0\in\cE$, $K\in L^2_{loc}(\bR_+,\bR^{d\times d})$ and, for all~$i\in\llbracket 1,d\rrbracket$ and~$w>0$, $K^i$ and~$R^w_i$ are non-negative, where~$R^w_i$ is the resolvent of the second kind of~$wK^i$.
\begin{enumerate}[(i)]
    \item The \textcolor{black}{stochastic evolution equation}~\eqref{eq:multiSPDE} admits a unique Markovian solution with values in~$\cE$ given by a generalised Feller semigroup~$(P_t)_{t\ge0}$ on~$\cB^\varrho(\cE)$.
The generator~$A:\bm F \to \cB^\varrho(\cE)$ associated to the semigroup~$(P_t)_{t\ge0}$ reads
\begin{equation}
   A f_y(\lambda) =  f_y(\lambda) \left( \langle \cA y,\lambda\rangle + \langle y b(\olambda),\nu\rangle
+ \half \langle y \sigma^\top(\olambda),\nu\rangle \langle y\sigma(\olambda),\nu\rangle \right)=:f_y(\lambda) \mathcal{R}(y,\lambda),
\label{eq:mygen}
\end{equation}
where the coefficients are given in~\eqref{eq:multicoefs} and $\cA$ is the adjoint operator of $\cA^\star$.
\footnote{This is the multidimensional version of the generator given in~\cite[(4.21)]{CT18}, where there seems to be a mild typo. We have adopted a slightly non-standard notation in that we first define ~$\cA^\star \lambda(\D x)=-x\lambda(\D x)$ and then set~$\cA$ to be the adjoint operator of~$\cA^\star$ mainly for reasons of being consistent with the existing literature because $\cA^\star$ is typically considered to act on processes living in $Y^{\star}$. }
    \item This generalised Feller process allows to construct a probabilistically weak and analytically mild solution, i.e. for~$y\in Y$,
    \[
    \langle y,\lambda_t \rangle = \langle y \E^{-t\cdot} , \lambda_0 \rangle + \int_0^t \langle y \E^{-(t-s)\cdot} , \nu \rangle \,\D X_s.
    \]
    \item The affine transform formula is satisfied, i.e.
    \[
    \bE_{\lambda_0}[\exp(\langle y_0,\lambda_t\rangle)] = \exp(\langle y_t,\lambda_0\rangle),
    \]
    where~$y_t$ solves
    \begin{equation}\label{eq:ODEyt}
    \langle \partial_t y_t,\lambda\rangle = \mathcal{R}(y_t,\lambda),
    \end{equation}
    for all~$\lambda\in\cE$, $y_0\in\cE_\star$, $t\ge0$, where~$\mathcal{R}$ is defined in~\eqref{eq:mygen}. Furthermore, $y_t\in\cE_\star$ for all~$t\ge0$.
    \item For any $\lambda_0 \in \cE$, the corresponding  stochastic Volterra equation given by
    \[
    \olambda_t=\langle 1,\lambda_t\rangle = \langle \E^{-t\cdot} , \lambda_0 \rangle + \int_0^t K(t-s)\,\D X_s
    \]
    admits a probabilistically weak solution.
    \item For all~$u\in\bR$, the Laplace transform of the Volterra equation~$\olambda_t$ is
    \[
    \bE_{\lambda_0}\left[\E^{u\olambda_t}\right] = \exp\left( u\langle \E^{-t\cdot},\lambda_0\rangle + \int_0^t \langle \E^{-(t-s)\cdot},\lambda_0\rangle  \left( \langle y_s b(\olambda_0),\nu\rangle
+ \half \langle y_s \sigma^\top(\olambda_0),\nu\rangle \langle y_s\sigma(\olambda_0),\nu\rangle \right)\ds\right).
    \]
\end{enumerate}
\end{proposition}

\textcolor{black}{Before we proceed with the proof of Proposition \ref{prop:GFP}, we would like to point out that even though we have calculated the generator $A f_y(\lambda) $ in Proposition \ref{prop:GFP} for Fourier basis elements, we have not calculated the generator for all elements of its domain of definition and very importantly we have not calculated its domain of definition. This is useful to keep in mind as without explicit knowledge of its domain of definition, it can be misleading to use the generator for drawing conclusions about limiting behaviour. As a matter of fact, in the case $K^i(t)=\frac{t^{\alpha(i)-1}}{\Gamma(\alpha(i))}\E^{-\delta t}$ with $V_{0}\neq 0$ and $\delta>0$, the calculation in (\ref{eq:ExpectationV_tinfinity}) shows that the limiting behaviour clearly depends on the initial condition $V_{0}\neq 0$.
}

\begin{proof}[Proof of Proposition \ref{prop:GFP}]
The entirety of~\cite[Section 4]{CT18} aims at proving~\cite[Theorem~4.17]{CT18}.
Our result follows from the same lines with a few modifications and remarks that we highlight here. The Lemmas, Assumptions, Remarks, Propositions and Theorems we refer to in this proof are all from~\cite{CT18}, unless stated otherwise.

\textbf{Modifications.} Since~$\E^{-t\cdot}\nu\in Y^\star$ for all~$t>0$ and
\[
\int_0^t \norm{\E^{-s\cdot}\nu}^2_{Y^\star} \ds = \int_0^t K(s)^2\ds<\infty,
\]
\cite[Assumption~4.5]{CT18} holds.
Moreover, \cite[Assumption~4.9]{CT18} is satisfied thanks to our assumption on~$K^i$ and~$R_i^w$ and Remark~4.10.

The space of signed measures is a vector space, hence a convex cone, and so are~$\cE^n$ for all $n\in\bN$. The weak-$\star$-continuity of the solution map and the bound of~$\varrho$ derived in Proposition~4.6 hold without modification. Essentially, every bound remains by first getting the bound for the autonomous one-dimensional square-root for~$i\in\llbracket 1,\tilde{d} \rrbracket$ and then plugging it in the other rows.

The invariant spaces~$\cE^n$ then satisfy all the necessary properties of \cite[Proposition~4.8 and Lemma~4.11]{CT18}  because they stay in their convex cone. This yields the generalised Feller property on~$\cB(\cE)$ of \cite[Theorem~4.13]{CT18}.

The variation of constants method continues to apply in multiple dimensions and for OU type processes, hence Proposition~4.14 still holds.

We recall that~$\cA^\star \lambda(\D x)=-x\lambda(\D x)$ for all~$\lambda\in\cM(\overline{\bR}_+,\bR^d)$.
In the proof of \cite[Theorem~4.17]{CT18}, for $i\in\llbracket1,\tilde{d}\rrbracket$, we replace \cite[Equation~(4.26)]{CT18} by
    \[
    \D\lambda^{i,n}_t = \cA^\star \lambda^{i,n}_t\dt + \nu \left(\beta_{ii} - n \sum_{j=1}^m \sigma_{iji}^2\right) \overline{\lambda}^{i,n}_t\dt + \nu \sum_{j=1}^m \D N^{ij,n}_t,
    \]
    where $N^{ij,n}_t$ is a jump process that jumps by~$1/n$ and with intensity~$n^2\sigma^2_{ij} \overline{\lambda}^{i,n}_t$. Similarly, for all~$i\in\llbracket \tilde{d}+1,d\rrbracket$,
    \[
    \D\lambda^{i,n}_t = \cA^\star \lambda^{i,n}_t \dt + \nu \left( \sum_{k=1}^d \beta_{ik} \overline{\lambda}^{k,n}_t\dt + \sum_{j=1}^m \sum_{k=1}^{\tilde{d}} \left(-n\left(\sigma_{ijk} \sqrt{\overline{\lambda}^{k,n}_s}+c_{ijk}\right)^2\dt + \D N^{ijk,n}_t  \right) \right),
    \]

    where $N^{ijk,n}_t$ is a jump process that jumps by~$1/n$ and with intensity~$n^2  \left(\sigma_{ijk}\sqrt{\overline{\lambda}^{k,n}_t}+c_{ijk}\right)^2$. The same arguments as in the proof of \cite[Theorem~4.17]{CT18} allow to conclude that the limit of~$\lambda^n$ as $n$ goes to infinity is a generalised Feller process which generator coincides with that of~\eqref{eq:multiSPDE} that we compute now.

\textbf{The generator.} Let us define, for any~$y\in \cE_\star$, $\lambda_0\in\cE$ and~$i\in\llbracket 1,d \rrbracket$,
\[
S_t^i:=\langle y^i,\lambda_t^i\rangle = \langle y^i\E^{-t\cdot} ,\lambda_0^i \rangle + \int_0^t \langle y^i\E^{-(t-s)\cdot} ,\nu^i\rangle \D X_s^i.
\]
This is not a semimartingale, therefore we introduce this approximation for any~$\ep>0$:
\[
S_t^{i,\ep}:= \langle y^i\E^{-t\cdot} ,\lambda_0^i \rangle + \int_0^t \langle y^i\E^{-(t+\ep-s)\cdot} ,\nu^i\rangle \D X_s^i.
\]
We define~$F(s):=\E^s$ such that~$F(S_t)=\E^{\langle y,\lambda_t \rangle}=\E^{\sum_i\langle y^i,\lambda_t^i \rangle}=f_{y}(\lambda_t)$ and by convergence in law $P_t f_{y}(\lambda_0) = \bE[F(S_t)]=\lim_{\ep\downarrow0} \bE[F(S^{\ep}_t)]$. By It\^{o}'s formula,
\[
F(S_t^{\ep})=F(S_0^{\ep}) + \sum_{i=1}^d \int_0^t \partial_i F(S_r^\ep) \D S_r^{i,\ep} + \half \sum_{i,j=1}^d \int_0^t \partial_{ij} F(S_r^\ep) \D[S^{i,\ep},S^{j,\ep}]_r,
\]
where, denoting~$\cA$ the adjoint of $\cA^\star$, we obtain
\begin{align*}
    \D S_r^{i,\ep} &= \left(\langle \cA y^i\E^{-r\cdot},\lambda_0^i \rangle + \langle y^i\E^{-\ep\cdot} ,\nu^i\rangle b_i(\olambda_r)
    + \int_0^r \langle \cA y^i \E^{(r+\ep-u)\cdot},\nu^i\rangle \D X_u^i \right)\dr
    + \langle y^i\E^{-\ep\cdot} ,\nu^i\rangle \sigma_i(\olambda_r)\,\D B_r.
\end{align*}
Hence, for all~$\lambda_0\in\cE$,
\begin{align*}
\bE[F(S_t^\ep)] = \E^{\langle y,\lambda_0 \rangle} + \int_0^t \bE\bigg[ &F(S^\ep_r) \sum_{i=1}^d\left( \langle \cA y^i\E^{-r\cdot},\lambda_0^i \rangle + \langle y^i\E^{-\ep\cdot} ,\nu^i\rangle b_i(\olambda_r)
    + \int_0^r \langle \cA y^i \E^{(r+\ep-u)\cdot},\nu^i\rangle b_i(\olambda_u) \du\right) \\
 +  &\frac{F(S^\ep_r)}{2} \sum_{i,j=1}^d \langle y^i\E^{-\ep\cdot} ,\nu^i\rangle  \langle y^j\E^{-\ep\cdot} ,\nu^j\rangle \sigma^\top_i\sigma_j(\olambda_r)
\bigg] \dr.
\end{align*}
Therefore,  we obtain \textcolor{black}{ for $\lambda_{0}$ the initial condition of the $\lambda$ process}
\begin{align*}
\lim_{t\downarrow0} \frac{P_t f_y(\lambda_0) - f_y(\lambda_{0})}{t}
&= \lim_{t\downarrow0} \frac{\lim_{\ep\downarrow0}\bE[F(S^\ep_t)] - f_y(\lambda_{0})}{t}\nonumber\\
&= f_y(\lambda_{0}) \sum_{i=1}^d\left( \langle \cA y^i ,\lambda^i_{0}\rangle + \langle y^i,\nu^i\rangle b_i(\olambda_{0})\right) +  \frac{f_y(\lambda_{0})}{2}\sum_{i,j=1}^d \left( \langle y^i,\nu^i\rangle \langle y^j,\nu^j\rangle \sigma^\top_i\sigma_j(\olambda_{0})\right) \nonumber\\
&= f_y(\lambda_{0}) \left( \langle \cA y,\lambda_{0}\rangle + \langle y b(\olambda_{0}),\nu\rangle
+ \half \langle y \sigma^\top(\olambda_{0}),\nu\rangle \langle y\sigma(\olambda_{0}),\nu\rangle \right),
\end{align*}
where the coefficients were given in~\eqref{eq:multicoefs}. This yields item (i), while item (ii) follows along the same lines as in \cite[Theorem~4.17]{CT18}.

By the existence of a generalised Feller semigroup, $\bE[\exp(\langle y_0,\lambda_t\rangle)]$ is the unique solution to the Cauchy problem
$$
\partial_t u(t,\lambda) = A u(t,\lambda), \qquad u(0,\lambda)=\exp(\langle y_0,\lambda\rangle).
$$
Moreover, for all $\lambda\in\cE$, there is a unique (mild) solution to
$$
\langle \partial_t y_t,\lambda\rangle = \mathcal{R}(y_t,\lambda), \quad y_0\in Y,
$$
as it is the adjoint equation to \eqref{eq:Equationeta}.
We have $f_{y_0}(\lambda) = \exp(\langle y_0,\lambda\rangle)$ and using the equation satisfied by $y_t$,
\[
    \partial_t f_{y_t}(\lambda) = f_{y_t}(\lambda)\langle \partial_t y_t,\lambda\rangle = f_{y_t}(\lambda) \mathcal{R}(y_t,\lambda) = A f_{y_t}(\lambda),
\]
    where the last equality follows from the definition of the generator in~\eqref{eq:mygen}.
Hence $f_{y_t}(\lambda)=\bE[\exp(\langle y_0,\lambda_t\rangle)]$.

The fourth point is a consequence of the second with~$y\equiv 1$.

    Finally, let~$y_0\equiv u\in\bR$, such that from (ii),
    \[
    \bE\left[\E^{u\olambda_t}\right]
    = \bE\left[\E^{\langle y_0,\lambda_t\rangle}\right]
    = \E^{\langle y_t,\lambda_0\rangle},
    \]
    where, in virtue of~\eqref{eq:ODEyt},
    \[
    \langle y_t,\lambda_0\rangle = u\langle \E^{-t\cdot},\lambda_0\rangle + \int_0^t \langle \E^{(t-s)\cdot},\lambda_0\rangle \left( \langle y_s b(\olambda_0),\nu\rangle
+ \half \langle y_s \sigma^\top(\olambda_0),\nu\rangle \langle y_s\sigma(\olambda_0),\nu\rangle \right)\ds.
    \]
    \notthis{which clearly indicates that $\psi_t=\langle y_t,\nu\rangle$ solves
    \[
    \psi_t = uK(t) + \int_0^t K(t-s)  \mathcal{R}(\psi_s)\ds.
    \]
    }
This concludes the proof of the proposition.
\end{proof}

\section{Verifying the condition for existence}\label{sec:bound}

We proved in the previous section that the solution of~\eqref{eq:multiSPDE} is indeed a generalised Feller process, hence all that is left to apply  Proposition~\ref{prop:existencecondition} is to show that the bound
\begin{equation}\label{eq:condition4}
\sup_{t\ge0} \bE_{\gamma}\left[\norm{\lambda_t}_{Y\star}\right] < \infty
\end{equation}
holds.
In this section, we specify the conditions on the parameters of the \textcolor{black}{stochastic evolution equation} that ensure~\eqref{eq:condition4}.
We start with some comments on the possible choices of initial condition~$\lambda_0$; recall that~$\lambda_0$ is admissible if it belongs to~$\cE$. By Assumption 5.2 and Remark~5.3 in~\cite{CT18}, $\E^{-r \cdot}\nu\in\cE$ for all $r>0$.
We can identify different conditions for initial conditions of OU rows or square-root rows.
Clearly, for~$i\in\llbracket \tilde{d}+1,d \rrbracket$, $\lambda_0^i$ can be any signed measure on~$\overline{\bR}$. For~$i\in\llbracket 1,\tilde{d} \rrbracket$, by Remark~5.4 in~\cite{CT18} and Example~2.2 in~\cite{JE18b}, are allowed all $\lambda_0\in Y^\star$ such that~$t\mapsto \int_0^\infty \E^{-tx} \lambda_0(\D x) \in \cG_K$, where~$\cG_K$ includes in particular H\"older continuous, non-decreasing functions~$g$ with~$g(0)\ge0$, and functions of the form~$g=V_0 + K \ast \theta$ such that~$\theta(s)\ds + V_0 L(\D s)$ is a non-negative measure, where~$V_0>0$ and~$L$ is the resolvent of first kind of~$K$.
This yields at least two options:
\begin{itemize}
    \item $\lambda_0(\D x)=V_0\, \delta_0(\D x)$, with~$V_0>0$;
    \item $\lambda_0(\D x) =V_0\, \delta_0(\D x)+  x^{-\alpha-\mu}\dx$ where~$0<\mu<1-\alpha$, such that
    \begin{align*}
    \int_0^\infty \E^{-xt}\lambda_0(\D x)
    & = V_0 +\Gamma(1-\alpha-\mu) t^{\alpha+\mu-1}\\
    & = V_0 + \frac{\Gamma(1-\alpha-\mu)\Gamma(\alpha+\mu)}{\Gamma(\alpha)\Gamma(\mu)} \int_0^t (t-s)^{\alpha-1} s^{\mu-1}\ds,
    \end{align*}
    which is equal to $V_0 + (K\ast\theta)(t)$ with $\theta$ and~$L$ non-negative.
\end{itemize}
We will be particularly interested in the case of constant initial condition because it allows us to compute $\bE[\olambda^k_t]$ explicitly, see~\eqref{eq:ExpectationV}, and because the bound~\eqref{eq:condition4} is independent of $\omega$.
This is the main result of this section.
\begin{proposition}\label{prop:bound}
Under Assumption~\ref{assu:main}, the bound \eqref{eq:condition4} holds.
\end{proposition}

We need to state an intermediate lemma, the proof of which is postponed to the end of the section, before being in a position to prove Proposition~\ref{prop:bound}.

\begin{lemma}\label{lemma:conditiondX}
The bound~\eqref{eq:condition4} holds if~$\norm{\lambda_0}_{Y^\star}<\infty$ and, for all~$i\in\llbracket 1,d \rrbracket$,
\begin{equation}\label{eq:conditiondX}
    \sup_{t\ge0}  \bE\left[   \int_0^\infty  \abs{\int_0^t \E^{-x(t-s)} \D X_s^i} \nu^i(\D x) \right] < \infty.
\end{equation}
\end{lemma}

\begin{remark}\label{rem:MarkovCase}
In the case where $\nu=\delta_0$ there is no kernel, hence this expectation boils down to~$\bE[\abs{X_t^i}]$ where~$X^i$ is a Markovian process.
For example, Condition \eqref{eq:conditiondX} holds if~$X$ is a multidimensional OU process, which is Gaussian with bounded variance, and if~$X$ is a one-dimensional square-root process, which is positive with bounded expectation.
\end{remark}

\begin{proof}[Proof of Proposition~\ref{prop:bound}]
First notice that~$\norm{\lambda_0}_{Y^\star}=V_0<\infty$.
For all~$i>\tilde{d}$, since $c_{ijk}=0$ and~$\beta_{ik}=0$ for~$k>\tilde{d}$,
\begin{align}
&\bE\left[  \int_0^\infty  \abs{\int_0^t \E^{-x(t-s)} \D X_s^i} \nu^i(\D x) \right] \nonumber \\
& \le \bE\left[  \sum_{k=1}^{\tilde{d}}\int_0^\infty\left(\int_0^t\E^{-x(t-s)} \beta_{ik}\olambda^k_s \ds \right)\nu^i(\D x) \right]
    +  \int_0^\infty \bE\left[ \left|\int_0^t\E^{-x(t-s)} \sum_{j=1}^m \sum_{k=1}^{\tilde{d}} \sigma_{ijk} \sqrt{\olambda^k_s}  \,\D B^j_s \right|\right]\nu^i(\D x) \label{eq:split}\\
&\le \sum_{k=1}^{\tilde{d}} \bE\left[\int_0^\infty\left(\int_0^t\E^{-x(t-s)} \beta_{ik}\olambda^k_s \ds \right)\nu^i(\D x) \right] + \int_0^\infty \bE\left[ \int_0^t\E^{-2x(t-s)} dm \sum_{j=1}^m \sum_{k=1}^{\tilde{d}} \sigma_{ijk}^2 \olambda^k_s \ds \right]^\half \nu^i(\D x), \nonumber
 \end{align}
by It\^o's isometry.
By Lemma~\ref{lemma:conditiondX} it suffices to prove that the latter is uniformly bounded in~$t\ge0$.
Recall that if~$k\le\tilde{d}$ then $\lambda^k$ is one-dimensional and, by~\cite[Lemma~4.2]{ALP17},
\begin{equation}\label{eq:ExpectationV}
    \bE[\olambda^k_t]= V_0^k\left( 1-  \beta_{kk}
    \int_{0}^{t}\E^{-\delta s}s^{\alpha-1} \mathrm{E}_{\alpha,\alpha}(-\beta_{kk} s^\alpha)\D s\right).
\end{equation}

Let us define $\widehat\nu^i(\D x):=x^{-\half}\nu^i(\D x)$. For any~$f:\bR_+\to\bR_+$, we can split the following integral and apply Jensen's inequality to both terms
\begin{align*}
    \int_0^\infty \sqrt{f(x)} \nu^i(\D x)
    &= \int_0^1 \sqrt{f(x)} \nu^i(\D x) + \int_1^\infty \sqrt{f(x) x}\, \widehat\nu^i(\D x) \nonumber\\
    &\le \sqrt{\nu^i((0,1))} \sqrt{\int_0^1 f(x)\, \nu^{i}(\D x)} + \sqrt{\widehat\nu^i((1,\infty))} \sqrt{\int_1^\infty f(x) x\, \widehat\nu^i(\D x)} \nonumber\\
    &\le \frac{\Gamma(1-\alpha(i))^{-1}}{1-\alpha(i)}  \sqrt{\int_0^\infty f(x) \nu^i(\D x)} + \frac{\Gamma(1-\alpha(i))^{-1}}{\alpha(i)-\half} \sqrt{\int_0^\infty f(x) x \,\widehat\nu^i(\D x)}.
\label{eq:splitzeroinf}
\end{align*}
We then set $f(x):=\bE\left[ \int_0^t \E^{-2x(t-s)} \sigma^2_{ijk} \olambda^k_s\ds \right]$.
Both terms lead to the same type of kernel, after an application of Fubini's theorem
\begin{align}
    \int_0^\infty \E^{-2x(t-s)} \nu^i(\D x) &= \int_0^\infty \E^{-2x(t-s)} \frac{1}{\Gamma(\alpha(i))\Gamma(1-\alpha(i))}(x-\delta)^{-\alpha(i)} \one_{x>\delta}\D x\nonumber\\
    &= \frac{1}{\Gamma(\alpha(i))\Gamma(1-\alpha(i))}\int_{\delta}^\infty \E^{-2x(t-s)} (x-\delta)^{-\alpha(i)} \D x\nonumber\\
    &= \frac{1}{\Gamma(\alpha(i))\Gamma(1-\alpha(i))}\E^{-2\delta(t-s)}\int_{0}^\infty \E^{-2y(t-s)} y^{-\alpha(i)} \D y\nonumber\\
    &= \frac{1}{\Gamma(\alpha(i))\Gamma(1-\alpha(i))}2^{\alpha(i)-1}(t-s)^{\alpha(i)-1}\E^{-2\delta(t-s)}\int_{0}^\infty \E^{-z} z^{-\alpha(i)} \D z\nonumber\\
    &=\frac{2^{\alpha(i)-1}}{\Gamma(\alpha(i))} (t-s)^{\alpha(i)-1}\E^{-2\delta(t-s)}.
    \end{align}
The same exact calculation yields
\begin{align}
    \int_0^\infty \E^{-2x(t-s)} x \, \widehat\nu^i(\D x) &= \int_0^\infty \E^{-2x(t-s)} \frac{1}{\Gamma(\alpha(i))\Gamma(1-\alpha(i))} (x-\delta)^{\half-\alpha(i)}\one_{x>\delta}  \D x\nonumber\\
    &= \frac{2^{\alpha(i)-\frac{3}{2}} \Gamma(\alpha(i)-\frac{3}{2})}{\Gamma(\alpha(i))} (t-s)^{\alpha(i)-\frac{3}{2}}\E^{-2\delta(t-s)}.\nonumber
\end{align}

We are left to consider integrals of the type
$\int_0^t \E^{-2\delta(t-s)} (t-s)^{\mu-1} \bE\left[\olambda_s^k\right]\ds$,
for~$\mu\in(0,1]$. We want to show that such quantities are uniformly bounded over $t\in[0,\infty)$.

For this purpose, we have via a change of the domain of integration
\begin{align}
&\int_0^t \E^{-2\delta(t-s)} (t-s)^{\mu-1} \bE\left[\olambda_s^k\right]\ds&=\nonumber\\
&=
V_0^k\left( \int_0^t \E^{-2\delta(t-s)} (t-s)^{\mu-1} \ds-  \beta_{kk} \int_0^t \E^{-2\delta(t-s)} (t-s)^{\mu-1}
    \int_{0}^{s}\E^{-\delta u}u^{\alpha-1} \mathrm{E}_{\alpha,\alpha}(-\beta_{kk} u^\alpha)\D u \D s\right)  \nonumber\\
&=
V_0^k\left( \int_0^t \E^{-2\delta(t-s)} (t-s)^{\mu-1} \ds-  \beta_{kk} \int_0^t\left(\int_{u}^{t} \E^{-2\delta(t-s)} (t-s)^{\mu-1}
    \E^{-\delta u}u^{\alpha-1} \mathrm{E}_{\alpha,\alpha}(-\beta_{kk} u^\alpha)\D s\right)\D u\right)  \nonumber\\
&=
V_0^k\left( \int_0^t \E^{-2\delta(t-s)} (t-s)^{\mu-1} \ds-  \beta_{kk} \int_0^t\left(\int_{u}^{t} \E^{-2\delta(t-s)} (t-s)^{\mu-1}
    \D s\right)\E^{-\delta u}u^{\alpha-1} \mathrm{E}_{\alpha,\alpha}(-\beta_{kk} u^\alpha)\D u\right)  \nonumber\\
 &=
V_0^k\left( \int_0^t \E^{-2\delta(t-s)} (t-s)^{\mu-1} \ds-  \beta_{kk} \int_0^t\left((2\delta)^{-\mu}\int_{0}^{2\delta(t-u)} \E^{-\kappa} \kappa^{\mu-1}
    \D \kappa\right)\E^{-\delta u}u^{\alpha-1} \mathrm{E}_{\alpha,\alpha}(-\beta_{kk} u^\alpha)\D u\right)  \nonumber\\
&=
V_0^k\left( \int_0^t \E^{-2\delta(t-s)} (t-s)^{\mu-1} \ds-  \beta_{kk} \int_0^t\left((2\delta)^{-\mu}\gamma(\mu,2\delta(t-u))\right)\E^{-\delta u}u^{\alpha-1} \mathrm{E}_{\alpha,\alpha}(-\beta_{kk} u^\alpha)\D u\right)  \nonumber\\
&=
V_0^k\left( (2\delta)^{-\mu}\gamma(\mu,2\delta t)-  \beta_{kk} \int_0^t\left((2\delta)^{-\mu}\gamma(\mu,2\delta(t-u))\right)\E^{-\delta u}u^{\alpha-1} \mathrm{E}_{\alpha,\alpha}(-\beta_{kk} u^\alpha)\D u\right)  \nonumber\\
&=
V_0^k\left( (2\delta)^{-\mu}\gamma(\mu,2\delta t)-  \beta_{kk} \int_0^{\infty}\one_{[0,t]}(u)\left((2\delta)^{-\mu}\gamma(\mu,2\delta(t-u))\right)\E^{-\delta u}u^{\alpha-1} \mathrm{E}_{\alpha,\alpha}(-\beta_{kk} u^\alpha)\D u\right)  \nonumber
    \end{align}
where $\gamma(\cdot,\cdot)$ denotes the incomplete Gamma function.

Now the quantity above is a  continuous functions in both $t\in[0,\infty)$ and $\delta\in(0,\infty)$. Next, we want to pass to the limit as $t\rightarrow\infty$ in the last expression.

For the first term we know that $\lim_{t\rightarrow\infty}\gamma(\mu,2\delta t)=\Gamma(\mu)$. For the second term we shall use dominated convergence. Define the function $u\mapsto\xi_{t}(u)$
\begin{align*}
\xi_{t}(u)&=\one_{[0,t]}(u)\left((2\delta)^{-\mu}\gamma(\mu,2\delta(t-u))\right)\E^{-\delta u}u^{\alpha-1} \mathrm{E}_{\alpha,\alpha}(-\beta_{kk} u^\alpha)
\end{align*}
 and notice that
\begin{align*}
|\xi_{t}(u)|&\leq \left((2\delta)^{-\mu}\Gamma(\mu)\right)\E^{-\delta u}u^{\alpha-1} \mathrm{E}_{\alpha,\alpha}(-\beta_{kk} u^\alpha)\in L^{1}([0,\infty)).
\end{align*}

Hence, by dominated convergence, we have that
 as $t\rightarrow\infty$ the quantity above  limits to, for $\delta$ fixed,
\begin{align}
&V_0^k\left( (2\delta)^{-\mu}\Gamma(\mu)-  \beta_{kk} \int_0^\infty\left((2\delta)^{-\mu}\Gamma(\mu)\right)\E^{-\delta u}u^{\alpha-1} \mathrm{E}_{\alpha,\alpha}(-\beta_{kk} u^\alpha)\D u\right) \nonumber\\
&=V_0^k (2\delta)^{-\mu}\Gamma(\mu)\left( 1-  \beta_{kk} \int_0^\infty \E^{-\delta u}u^{\alpha-1} \mathrm{E}_{\alpha,\alpha}(-\beta_{kk} u^\alpha)\D u\right) \nonumber\\
&=V_0^k (2\delta)^{-\mu}\Gamma(\mu)\left( 1-  \frac{\beta_{kk}}{\beta_{kk}+\delta^{\alpha}} \right) \nonumber\\
&<\infty,
\end{align}
the last identity following by the integral representation properties of the Mittag-Leffler function, see for example \cite[Section~7]{Haubold2011}.
Hence, we indeed get the uniform finiteness in $t\in[0,\infty)$ for $i>\tilde{d}$.

On the other hand, if~$i<\tilde{d}$,
\begin{align*}
\bE\left[  \int_0^\infty  \abs{\int_0^t \E^{-x(t-s)} \D X_s^i} \nu^i(\D x) \right]
\le \, & \bE\left[ \int_0^\infty\left(\int_0^t\E^{-x(t-s)} \beta_{ii}\olambda^i_s \ds \right)\nu^i(\D x) \right]\\
&    +  \int_0^\infty \bE\left[ \Big\lvert\int_0^t\E^{-x(t-s)} \sum_{j=1}^m \sigma_{iji} \sqrt{\olambda^i_s}  \,\D B^j_s \Big\lvert\right]\nu^i(\D x),
\end{align*}
which boils down to~\eqref{eq:split}.

\end{proof}

\begin{remark}\label{rem:Bxt}
The case~$c_{ijk}>0$ is unlikely to go through by splitting the integrals as we did here.
Consider~\eqref{eq:conditiondX} where~$X$ is a one-dimensional OU process and split the Lebesgue and It\^o integrals as in~\eqref{eq:split}. The stochastic integral then reads
\begin{equation}
    \int_0^\infty \bE[\abs{B^x_t}]  x^{-\alpha}\dx,
\label{eq:Bxt}
\end{equation}
where~$B^x_t := \int_0^t \E^{-x(t-s)} \,\D W_s$. Since~$\displaystyle\bE[\abs{B^x_t}] = \sqrt{\frac{1-\E^{-2xt}}{x\pi}}$, this entails that~\eqref{eq:Bxt} grows as~$t^{\alpha-\half}$, and thus is not uniformly bounded.
\end{remark}

We conclude with the proof of Lemmas~\ref{lemma:conditiondX}.
\begin{proof}[Proof of Lemma~\ref{lemma:conditiondX}]
Let $\mathfrak B(\overline\bR_+)$ denote the Borel subsets of $\overline\bR_+$.
For a signed measure~$\lambda$ on~$\overline\bR_+$, define the upper variation~$\mathfrak U$ (resp. lower variation~$\mathfrak L$) as~$\mathfrak U(\lambda):=\sup\{\lambda(A):A\in\mathfrak B(\overline\bR_+)\}$ (resp.~$\inf$), and such that the total variation corresponds to~$\norm{\lambda}_{Y^\star}=\mathfrak U(\lambda)-\mathfrak L(\lambda)$.

Let us fix~$t>0$, then there exist two increasing sequences of sets $(U_n)_{n\ge1}$ and~$(L_n)_{n\ge1}$, both in~$\mathfrak B(\overline\bR_+)$, such that~$\lambda_t(U_n)$ is non-negative for all~$n\in\bN$ and increases towards~$\mathfrak U(\lambda_t)$ as~$n$ goes to $+\infty$, and analogously,  $\lambda_t(L_n)$ is non-positive and decreases to~$\mathfrak L(\lambda_t)$.
We will use the representation
\begin{equation}\label{eq:rpznormlambda}
\norm{\lambda_t}_{Y^\star}
=\lim_{n\uparrow\infty} \big(\lambda_t(U_n)-\lambda_t(L_n)\big)
=\lim_{n\uparrow\infty} \left( \int_{U_n}\lambda_t(\D x) - \int_{L_n} \lambda_t(\D x)\right).
\end{equation}
Intuitively, one should think of the limit of~$U_n$ (resp. $L_n$) as the subset of $\overline\bR_+$ on which~$\lambda_t$ is positive (resp. negative). This way, the function maximising $\sup_{\norminf{y}\le 1} \langle y,\lambda_t\rangle$ (if it exists) corresponds to the limit of $\one_{U_n}-\one_{L_n}$.

We did not assume that $\norm{\lambda_t}_{Y^\star}$ should be finite.
Yet, we can apply the monotone convergence theorem on the representation~\eqref{eq:rpznormlambda}
\begin{align}
    \bE[\norm{\lambda_t}_{Y^\star}]
    &= \lim_{n\uparrow+\infty} \bE\left[\int_{U_n}\lambda_t(\D x) - \int_{L_n} \lambda_t(\D x)\right] \nonumber\\
    & = \lim_{n\uparrow+\infty} \bE\left[ \int_0^\infty (\one_{U_n}-\one_{L_n}) \left( \E^{-xt}\lambda_0(\D x) + \int_0^t \E^{-x(t-s)} \,\D X_s \,\nu(\D x) \right) \right].
\label{eq:totalvar}
\end{align}
Since~$\norm{\lambda_0}_{Y^\star}<\infty$ by assumption, the first term is finite.
Starting from~\eqref{eq:totalvar}, we consider the worst possible sets~$U_n$ and~$L_n$, recall that~$\nu$ is diagonal and the form of~$X$:
\begin{align*}
     & \bE\left[ \int_0^\infty  (\one_{U_n}-\one_{L_n})(x) \left(\int_0^t \E^{-x(t-s)} \D X_s\right) \nu(\D x) \right]\\
    &= \bE\left[ \sum_{i=1}^d  \int_0^\infty  (\one_{U_n}-\one_{L_n})(x) \left(\int_0^t \E^{-x(t-s)} \D X_s^i\right) \nu^i(\D x) \right]\nonumber\\
    &\le \sum_{i=1}^d \bE\left[   \int_0^\infty  \abs{\int_0^t \E^{-x(t-s)} \D X_s^i} \nu^i(\D x) \right], 
\end{align*}
which yields the claim.
\end{proof}

\section{Outlook}\label{S:Outlook}
It is of interest to extend the results to more general Volterra processes. By shifting the generator of the lift it is possible to add a drift in the coefficient, for instance in the one-dimensional case~$b(x)=\beta(\theta-x)$, and prove that this still produces a generalised Feller process. We have investigated the validity of the Condition~\eqref{eq:conditionrho} in this case, but it is not clear at this point how to complete the proof.

Building on the generalised Feller property, we also aim at showing the uniqueness of the invariant measure in the appropriate space, excluding the zero process. The analogue of the strong Feller property is granted in our setting since~$\cB^\varrho(\bm X)$ already includes the space of bounded functions; uniqueness shall follow if~$(\lambda_t)_{t\ge0}$ satisfies a sort of recurrence property. One can also hope to characterise the unique invariant measure using tools from Malliavin calculus or the form of the Laplace transform. Once ergodicity is established, it is interesting to identify the rate of convergence and how it varies with~$\alpha$, thereby linking roughness of the volatility and convergence to the stationary measure.

\section*{Acknowledgement}
The authors would like to thank the reviewer for a very careful review of the manuscript that identified an issue with the initial version of the result and for allowing us to address it.

\section*{Funding}
A.P. and A.J. are
supported by an EPSRC EP/T032146/1 grant.
K.S. is partially supported by the National Science Foundation (DMS 2107856) and Simons Foundation Award  672441.

\bibliographystyle{abbrv}
\bibliography{bib}

\end{document}